\documentclass{article}
\usepackage{amsmath,amssymb,amsthm,amsfonts}
\makeatletter
\renewcommand\section{\@startsection {section}{1}{\z@}%
                                   {-3.5ex \@plus -1ex \@minus -.2ex}%
                                   {2.3ex \@plus.2ex}%
                                   {\normalfont\Large\bfseries
                                    \setcounter{equation}{0}
                                    \setcounter{thm}{0}}}
\makeatother
\newtheorem{thm}{Theorem}
\newtheorem{lem}[thm]{Lemma}
\newtheorem{cor}[thm]{Corollary}
\newtheorem{defn}[thm]{Definition}

\renewenvironment{proof}{{\noindent\bfseries{Proof.}}}{\hfill\ensuremath{\Box}}

\newcommand{\q}{\quad}
\newcommand{\qq}{\qquad}

\newcommand{\CommaBin}{\mathbin{\raisebox{0.5ex}{,}}}
\newcommand{\s }{\\[0.2cm]}
\title{\textbf{Well-Posedness of some nonlinear Volterra-Fredholm integral and integro-dynamic equations on time scales}}

\author{\small \textbf{Alaa E.~Hamza}$^{1}$, \textbf{and Ahmed G.~Ghallab}$^{2}$  \\
\small$^\textbf{{1}}$ Department of Mathematics, Faculty of Science, Cairo University, Giza, Egypt.\\
\small E-mail: hamzaaeg2003@yahoo.com\\
\small$^\textbf{{2}}$ Department of Mathematics, Faculty of  Science,
 \small {Fayoum University, Fayoum, Egypt.}\\
 \small E-mail: agg00@fayoum.edu.eg}
\date{}

\begin{document}
\thispagestyle{empty} \maketitle
\begin{abstract}
In this paper we study well posedness of a certain nonlinear Volterra-Fredholm dynamic integral and integro-dynamic equations on unbounded interval from arbitrary time scale. We derive the time scale analogue of certain integral inequalities of Pachpatte type and using them with Banach's fixed-point theorem to establish the results.
\end{abstract}
\let\thefootnote\relax\footnote{The results of this article appeared in a preliminary form as a scientific poster in the proceedings of the first international conference, new horizons in basic and applied science (ICNHBAS), 21-23 September, 2013, Hurghada, Egypt.}

\section{Introduction}
 Recently, numerous mathematicians have investigated some of the qualitative and the quantitative properties of different types of integral equations on time scales by using different techniques, see \cite{D3, D5, Kul, D1, D2, D4}. Integral equations on time scales are thought to have a great potential for many applications in various areas of natural science and to give a deeper understanding than the traditional integral and summation equations to many phenomena, especially those that occur mutually in continuous-discrete manner with the flow of time, see \cite{wong}.
\newpage
In the present paper we shall consider the general nonlinear dynamic integral equation
\begin{equation}\label{e1}
x(t)=f\Big(t,x(t),\int_a^t h(t,s,x(s))\Delta s,\int_a^b g(t,s,x(s))\Delta s\Big),\q t\in I_\mathbb{T},
\end{equation}
and the integro-dynamic equation
\begin{equation}\label{e.2}
x^\Delta(t)=f\Big(t,x(t),\int_a^t h(t,s,x(s))\Delta s,\int_a^b g(t,s,x(s))\Delta s\Big),\q t\in I_\mathbb{T},
\end{equation}
where $a<b$, and $I_\mathbb{T}=[a,\infty)_\mathbb{T}:=[a,\infty)\cap \mathbb{T}$. A time scale $\mathbb{T}$ is a nonempty closed subset from the reals. We assume that $f:I_\mathbb{T}\times \mathbb{X}\times \mathbb{X}\times \mathbb{X}\rightarrow \mathbb{X}$, is rd-continuous in the first variable, while the functions $h:I_\mathbb{T}^2\times \mathbb{X}\rightarrow \mathbb{X}$, $g:I_\mathbb{T}^2\times \mathbb{X}\rightarrow \mathbb{X}$ are assumed to be rd-continuous in the second argument, and $x$ is the unknown function. Here $\mathbb{X}$ is a Banach space. The integral is the delta integral and $\Delta$ denotes the delta derivative, for more details see \cite{Boh1, Boh2}. By a solution of \eqref{e1} (resp.\ Eq.\eqref{e.2}) we mean rd-continuous function $x:I_\mathbb{T}\rightarrow\mathbb{X}$ that satisfies Eq.\eqref{e1} (resp.\ Eq.\eqref{e.2}).

 In this paper, we shall investigate the well-posedness of equations \eqref{e1} and \eqref{e.2}. We shall study existence and uniqueness of the solutions, besides continuous dependence on data and parameters.


Our methods involve using Banach's fixed-point theorem on an appropriate metric space and deriving the time scale analogue of certain integral inequalities of Pachpatte type to study some qualitative properties of equations \eqref{e1} and \eqref{e.2}. The results of this paper generalize some of the results in articles \cite{akram, D1, D2, D4} on $\mathbb{T}=\mathbb{R}$.

The work in this paper is organized as follows. Section $\mathbf{3}$ establishes the time scale analogue of some integral inequalities of Pachpatte type. In Section $\mathbf{4}$, we prove existence and uniqueness of solutions of equations \eqref{e1} and \eqref{e.2}, respectively. In Section $\mathbf{5}$, we establish estimates on the solutions of equations \eqref{e1} and \eqref{e.2}, respectively. Section $\mathbf{6}$ investigates the continuous dependence of the solutions of equations \eqref{e1} and \eqref{e.2} on the functions involved.

\section{Preliminaries}
In this section we introduce some definitions, notations, and preliminary results which will be used throughout this paper.

\begin{defn}
A time scale $\mathbb{T}$ is a nonempty closed subset of the real numbers $\mathbb{R}$.
\end{defn}


\begin{defn}
The mappings $\sigma,\rho:\mathbb{T}\rightarrow \mathbb{T}$ defined by $\sigma(t)=\inf\{s\in \mathbb{T}: s>t\}$, and $\rho(t)=\sup\{s\in \mathbb{T}: s<t\}$ are called the jump operators.
\end{defn}
If $\mathbb{T}$ has a left scattered maximum $m$, then define $\mathbb{T}^\kappa=\mathbb{T}-\{m\}$, otherwise $\mathbb{T}^\kappa=\mathbb{T}$.

\begin{defn}
A function $f:\mathbb{T} \rightarrow \mathbb{X}$ is said to be delta differentiable at the point $t\in \mathbb{T}^\kappa$ if there exist an element $f^\Delta(t)\in \mathbb{X}$ with the property that given any $\epsilon >0$ there is a neighborhood $U$ of $t$ with $\|[f(\sigma(t))-f(s)]-f^\Delta(t)[\sigma(t)-s]\|\leq \varepsilon |\sigma(t)-s|$ for all $s\in U$. The function $f^\Delta(t)$ is the delta derivative of $f$ at $t$.
\end{defn}

For $\mathbb{T}=\mathbb{R}$, we have $f^\Delta(t)=f'(t)$, and for $\mathbb{T}=\mathbb{Z}$, we have $f^\Delta(t)=\Delta f(t)= f(t+1)-f(t)$.

\begin{defn}
A function $F:\mathbb{T}\rightarrow\mathbb{X}$ is called an \emph{antiderivative} of $f:\mathbb{T}\rightarrow\mathbb{X}$ provided $F^\Delta(t)=f(t)$ for all $t\in \mathbb{T}$. The $\Delta$-integral of $f$ is defined  by
$$
\int_r^s f(t)\Delta t= F(s)-F(r), \q \text{for all } r,s \in \mathbb{T}.
$$
\end{defn}

If $\mathbb{T}=\mathbb{R}$, then $\displaystyle\int_a^b f(s)\Delta s=\displaystyle\int_a^b f(s)ds$, while if $\mathbb{T}=\mathbb{Z}$, then $\displaystyle\int_a^b f(s)\Delta s=\displaystyle\sum_{s=a}^{b-1}f(s)$.

\begin{defn}
A function $f:\mathbb{T}\rightarrow\mathbb{X}$ is called  right-dense continuous (rd-continuous)  if $f$ is continuous at every right-dense point $t\in \mathbb{T}$ and the left-sided limits exist (i.e finite) at every left-dense point $t\in \mathbb{T}$. The family of all rd-continuous functions from $ \mathbb{T}$ to $\mathbb{X}$ is denoted by $ C_{rd}(\mathbb{T};\mathbb{X})$.
\end{defn}

The family  of all regressive functions is denoted  by
$$
\mathcal{R}:= \Big\{p\in C_{rd}(\mathbb{T};\mathbb{R}) \ \text {and } 1+p(t)\mu(t)\neq 0, \ \forall \ t\in \mathbb{T}\Big\}\CommaBin
$$
and the set of positively regressive functions is denoted by
$$
\mathcal{R}^+:= \Big\{p\in C_{rd}(\mathbb{T};\mathbb{R}) \ \text {and } 1+p(t)\mu(t)> 0, \ \forall \ t\in \mathbb{T}\Big\}\cdot
$$
\begin{defn}
If $p\in\mathcal{R}$, then we define the generalized exponential function by
$$
e_p(t,s)=\exp\Big( \int_s^t \xi_{\mu(\tau)}(p(\tau))\Delta \tau \Big)\q \text{for } \ t,s \in \mathbb{T},
$$
with the cylinder transformation
  \begin{equation*}
\xi_h (z)= \begin{cases}
   \displaystyle\frac{\log(1+hz)}{h} &\text{if } h\neq 0 \\[12pt]
z & \text{if } h=0.
\end{cases}
\end{equation*}
\end{defn}

In the case $\mathbb{T}=\mathbb{R}$, the exponential function is given by
$$
e_p(t,s)= \exp\Big ( \int _s^t p(\tau)d\tau \Big),
$$
for $s,t \in \mathbb{R}$, where $p:\mathbb{R}\rightarrow\mathbb{R}$ is a continuous function. In the case  $\mathbb{T}=\mathbb{Z}$, the exponential is given by
$$
e_p(t,s)=\prod_{\tau=s}^{t-1}[1+p(\tau)],
$$
for $s,t \in \mathbb{Z}$, where $p:\mathbb{Z}\rightarrow\mathbb{R}$, $ p(t)\neq -1$ for all $t\in \mathbb{Z}$.\\
For more basic properties of the generalized exponential function, see \cite{Boh1}.

\begin{defn}
For $k(x,y)\in \mathcal{R}$ with respect to the $y$, the generalized exponential function is defined by
$$
e_{k(x,.)}(t,s)=\exp(\int_s^t \xi_{\mu(\tau)}(k(x,\tau))\Delta \tau),
$$
for $s,t\in \mathbb{T}$.
\end{defn}

Let $\beta > 0$ be a constant and let $\|.\|$ denotes a norm on $\mathbb{X}$. We consider the space $C_\beta(I_\mathbb{T};\mathbb{X})$ of all rd-continuous functions such that
$$
\sup_{t\in I_\mathbb{T}}\frac{\|x(t)\|}{e_\beta(t,a)}<\infty,
$$
coupled  with a suitable norm, namely
$$
\|x\|_\beta^\infty=\sup_{t\in I_\mathbb{T}}\frac{\|x(t)\|}{e_\beta(t,a)}\cdot
$$
We can follow the proof of Lemma 3.3 in \cite{Tis} to obtain the following result.
\begin{lem} If $\beta>0$ is a constant, then
$(C_\beta(I_\mathbb{T};\mathbb{X}),\|\cdot\|_\beta^\infty)$ is a Banach space.
\end{lem}
We use the following comparison lemma in our study, see \cite{Boh1}.
\begin{lem}\label{lem2.2}
Suppose $u,\beta\in C_{rd}(I_\mathbb{T}, \mathbb{R})$ and $\alpha\in \mathcal{R}^+$. If
$$
u^\Delta(t)\leq \alpha(t)u(t)+\beta(t),\q \forall\  t\in I_\mathbb{T},
$$
then
$$
u(t)\leq u(a)e_\alpha(t,a)+\int_{a}^te_\alpha(t,\sigma(\tau))\beta(\tau)\Delta \tau, \q \forall \ t\in I_\mathbb{T}.
$$
\end{lem}

\section{New Pachpatte type inequalities on time scales}
In this section, we establish the time scale analogue of the integral inequalities of Pachpatte type given in \cite{pach}.
\begin{thm}\label{th3.1}
Let the functions $k_i(\cdot,\cdot):I_\mathbb{T}\times I_\mathbb{T}\rightarrow\mathbb{R}_+$ ( for $i=1,2$ ) be nondecreasing in the first variable and rd-continuous in the second variable. Also, assume that $u\in C_{rd}(I_\mathbb{T},\mathbb{R}_+)$ and
\begin{equation}
u(t)\leq c+ \int_a^tk_1(t,s)u(s)\Delta s + \int_a^bk_2(t,s)u(s)\Delta s, \q t\in I_\mathbb{T},
\end{equation}
 where $c\geq0$ is a real constant. If
\begin{equation}\label{cond1}
  p=\int_a^bk_2(b,s)e_{k_1(s,.)}(s,a)\Delta s <1,\q  t\in I_\mathbb{T},
\end{equation}
then
\begin{equation}\label{ineq}
 u(t)\leq \frac{c}{1-p}e_{k_1(t,.)}(t,a),\q  t\in I_\mathbb{T}.
\end{equation}
\end{thm}
\begin{proof}
Fix any $T\in I_\mathbb{T}$. Then for $t\in[a,T]_{\mathbb{T}}$ we have
\begin{equation}\label{e6}
u(t)\leq c+ \int_a^tk_1(T,s)u(s)\Delta s + \int_a^bk_2(T,s)u(s)\Delta s.
\end{equation}
Define the function $z$ by
\begin{equation}\label{eq16}
  z(t,T)= c+ \int_a^tk_1(T,s)u(s)\Delta s + \int_a^bk_2(T,s)u(s)\Delta s, \q   t\in[a,T]_\mathbb{T}.
\end{equation}
Then
\begin{equation}\label{eq06}
  u(t)\leq z(t,T), \q   t\in[a,T]_\mathbb{T},
\end{equation}
and
\begin{equation}\label{e10}
z(a,T)=c+\int_a^bk_2(T,s)u(s)\Delta s.
\end{equation}
By differentiating \eqref{eq16} with respect to $t$, we obtain
\begin{equation}\label{e7}
z^\Delta(t,T)=k_1(T,t)u(t)\leq k_1(T,t)z(t,T), \q  t\in[a,T]_\mathbb{T}.
\end{equation}
Since $k_1(T,t)\in \mathcal{R}^+$, we  apply Lemma \ref{lem2.2} with $\beta=0$ to obtain
\begin{equation}\label{e8}
 z(t,T)\leq z(a,T)e_{k_1(T,.)}(t,a), \q  t\in[a,T]_\mathbb{T}.
\end{equation}
Since $T$ is selected from $I_\mathbb{T}$ arbitrarily, we replace $T$ by $t$ in \eqref{eq06}, \eqref{e10}, and \eqref{e8}. So, we have
\begin{equation}\label{e9}
u(t)\leq z(t,t)\leq z(a,t)e_{k_1(t,.)}(t,a), \q  t\in I_\mathbb{T},
\end{equation}
where
\begin{equation}\label{e10.0}
z(a,t)=c+\int_a^bk_2(t,s)u(s)\Delta s,\q  t\in I_\mathbb{T}.
\end{equation}
Since $z(a,t)$ is nondecreasing in $t$, then
$$
z(a,t)\leq z(a,b), \q t\in I_\mathbb{T},
$$
and
\begin{equation} \label{nn}
u(t)\leq z(a,b)e_{k_1(t,.)}(t,a), \q t\in I_\mathbb{T}.
\end{equation}
From \eqref{e10.0} and \eqref{nn}, we have
\begin{align*}
z(a,b)&= c+\int_a^bk_2(b,s)u(s)\Delta s\\
&\leq c+\int_a^b k_2(b,s)z(a,b)e_{k_1(s,.)}(s,a)\Delta s\\
&= c+z(a,b)\int_a^b k_2(b,s)e_{k_1(s,.)}(s,a)\Delta s.
\end{align*}
In view of condition \eqref{cond1} it is easy to observe that
\begin{equation}\label{e11}
  z(a,b)\leq \frac{c}{1-p}\CommaBin \q  t\in I_{\mathbb{T}}.
\end{equation}
Applying \eqref{e11} in \eqref{nn} we obtain the desired inequality.
\end{proof}


\begin{thm}\label{th3.2}
Let $u,f,g,h\in C_{rd}(I_\mathbb{T},\mathbb{R}_+)$. Assume
\begin{equation}\label{e17.0}
 u(t)\leq k+\int_a^tf(s)\Big[u(s)+\int_a^sg(\tau)u(\tau)\Delta\tau+\int_a^bh(\tau)u(\tau)\Delta\tau\Big]\Delta s,\q  t\in I_{\mathbb{T}},
 \end{equation}
where $k\geq0$ is a real constant. If
\begin{equation}\label{e18}
  r=\int_a^bh(\tau)e_{f+g}(\tau,a)\Delta \tau < 1,
\end{equation}
then
\begin{equation}\label{e180}
  u(t)\leq \frac{k}{1-r}e_{f+g}(t,a),\q  t\in I_{\mathbb{T}}.
\end{equation}
\end{thm}
\begin{proof}
Define a function $z(t)$ by the right hand side of \eqref{e17.0}. Then $z(a)=k$, $u(t)\leq z(t)$ and
\begin{align*}
z^\Delta&=f(t)\Big[u(t)+\int_a^tg(\tau)u(\tau)\Delta\tau+\int_a^bh(\tau)u(\tau)\Delta\tau\Big]\\
&\leq f(t)\Big[z(t)+\int_a^tg(\tau)z(\tau)\Delta\tau+\int_a^bh(\tau)z(\tau)\Delta\tau\Big]
\end{align*}
for $t\in I_\mathbb{T}$. Define $w(t)$ by
$$
w(t)= z(t)+\int_a^tg(\tau)z(\tau)\Delta\tau+\int_a^bh(\tau)z(\tau)\Delta\tau, \q t\in I_\mathbb{T}.
$$
Then $z(t)\leq w(t)$, $z^\Delta\leq f(t)w(t)$. We have
\begin{equation}\label{18.0}
w(a)=k+\int_a^bh(\tau)z(\tau)\Delta\tau,
\end{equation}
and
$$
w^\Delta(t)=z^\Delta(t)+g(t)z(t)\leq f(t)w(t)+g(t)z(t)\leq [f(t)+g(t)]w(t).
$$
This implies, by Lemma \ref{lem2.2}, that
$$
w(t)\leq w(a)e_{f+g}(t,a), \q t\in I_\mathbb{T}.
$$
Since $z(t)\leq w(t)$, we get
\begin{equation}\label{18.00}
  z(t)\leq w(a)e_{f+g}(t,a), \q t\in I_\mathbb{T}.
\end{equation}
Using \eqref{18.00} on the right hand side of \eqref{18.0}, we obtain
$$
w(a)\leq k+w(a)\int_a^bh(\tau)e_{f+g}(\tau,a)\Delta \tau.
$$
In view of \eqref{e18} we get
\begin{equation}\label{18.000}
  w(a)\leq \frac{k}{1-r}\cdot
\end{equation}
Combining \eqref{18.000}, \eqref{18.00} with the inequality $u(t)\leq z(t)$, we get the desired inequality.
\end{proof}


\section{Existence and Uniqueness of Solutions}
 Banach's fixed point theorem is a powerful and important tool in providing sufficient conditions for existence and uniqueness of solutions of dynamic equations as well as integral equations on time scales, see \cite{Kul, Tis}. In this section, we use Banach's fixed point theorem to prove the existence and uniqueness of solutions of equations \eqref{e1} and \eqref{e.2}, respectively.
\begin{thm}\label{th1}
 Consider the integral equation \eqref{e1}. Suppose that there exist non-negative constants $M,L,N$, and $\gamma>1$ such that the following conditions are satisfied
\begin{equation}\label{e2}
  \|f(t,u,v,w)-f(t,\bar{u},\bar{v},\bar{w})\|\leq M\{\|u-\bar{u}\|+\|v-\bar{v}\|+\|w-\bar{w}\|\},
\end{equation}
\begin{equation}\label{e4}
  \|h(t,s,v)-h(t,s,\bar{v})\|\leq L\|v-\bar{v}\|,
\end{equation}
\begin{equation}\label{e3}
  \|g(t,s,u)-g(t,s,\bar{u})\|\leq N\|u-\bar{u}\|,
\end{equation}
\begin{equation}\label{e}
M(1+\frac{1}{\gamma})<1.
\end{equation}
In addition, assume that
\begin{equation}\label{e5}
  A_1:=\sup_{t\in I_\mathbb{T}}\frac{1}{e_\beta(t,a)}\Big\|f(t,0,\int_a^th(t,s,0)\Delta s,\int_a^bg(t,s,0)\Delta s)\Big\|<\infty,
\end{equation}

where $\beta$ is the solution of the equation $\beta:=\gamma(L+Ne_\beta(b,a))$. Then the integral equation \eqref{e1} has a unique solution $x\in C_\beta(I_\mathbb{T};\mathbb{X})$.
\end{thm}
\begin{proof}
Consider the Banach space $(C_\beta(I_\mathbb{T};\mathbb{X}),\|\cdot\|_\beta^\infty)$. Define the operator $F:C_\beta(I_\mathbb{T};\mathbb{X})\rightarrow C_\beta(I_\mathbb{T};\mathbb{X})$ by
\begin{equation*}
[Fx](t):=f\Big(t,x(t),\int_a^t h(t,s,x(s))\Delta s,\int_a^b g(t,s,x(s))\Delta s\Big),
\end{equation*}
for $t\in I_\mathbb{T}$. Fixed points of $F$ will be solutions to \eqref{e1}. Now, we prove that $F$ maps $C_\beta(I_\mathbb{T};\mathbb{X})$ into itself. Let $x\in C_\beta(I_\mathbb{T};\mathbb{X})$ and using the hypotheses, we have
\begin{align*}
 \|Fx\|_\beta^\infty
& = \sup_{t\in I_\mathbb{T}}\frac{\|(Fx)(t)\|}{e_\beta(t,a)}\\
& \leq \sup_{t\in I_\mathbb{T}}\frac{1}{e_\beta(t,a)}\Big\|f\Big(t,x(t),\int_a^t h(t,s,x(s))\Delta s,\int_a^b g(t,s,x(s))\Delta s\Big)\\
&\qq \qq \ -f\Big(t,0,\int_a^t h(t,s,0)\Delta s,\int_a^b g(t,s,0)\Delta s\Big)\Big\|\\
&\q +\sup_{t\in I_\mathbb{T}}\frac{1}{e_\beta(t,a)}\Big\|f\Big(t,0,\int_a^t h(t,s,0)\Delta s,\int_a^b g(t,s,0)\Delta s\Big)\Big\|\\
&\leq A_1+\sup_{t\in I_\mathbb{T}}\frac{1}{e_\beta(t,a)}M\Big\{\|x(t)\|+L\int_a^t\|x(s)\|\Delta s + N\int_a^b\|x(s)\|\Delta s\Big\}\\
& = A_1+M\Big\{ \sup_{t\in I_\mathbb{T}}\frac{\|x(t)\|}{e_\beta(t,a)}+L\sup_{t\in I_\mathbb{T}}\frac{1}{e_\beta(t,a)}\int_a^t\frac{\|x(s)\|}{e_\beta(s,a)}e_\beta(s,a)\Delta s \\
&\qq \qq \qq \q + N\sup_{t\in I_\mathbb{T}}\frac{1}{e_\beta(t,a)}\int_a^be_\beta(s,a)\frac{\|x(s)\|}{e_\beta(s,a)}\Delta s\Big\}\\
& \leq A_1+M\Big\{\|x\|_\beta^\infty+L \|x\|_\beta^\infty\sup_{t\in I_\mathbb{T}}\frac{1}{e_\beta(t,a)}\int_a^te_\beta(s,a)\Delta s\\
&\qq \qq \qq + N\|x\|_\beta^\infty \sup_{t\in I_\mathbb{T}}\frac{1}{e_\beta(t,a)} \int_a^be_\beta(s,a)\Delta s\Big\}\\
&\leq A_1+M\|x\|_\beta^\infty\Big\{1+L\sup_{t\in I_\mathbb{T}}\frac{1}{e_\beta(t,a)}\Big(\frac{e_\beta(t,a)-1}{\beta}\Big)+N\sup_{t\in I_\mathbb{T}}\frac{1}{e_\beta(t,a)}\Big(\frac{e_\beta(b,a)-1}{\beta}\Big)\Big\}\\
&\leq A_1+M\|x\|_\beta^\infty\Big\{1+\frac{L}{\beta}\sup_{t\in I_\mathbb{T}}\Big(1-\frac{1}{e_\beta(t,a)}\Big)
+\frac{N}{\beta}\frac{1}{e_\beta(a,a)}(e_\beta(b,a)-1)\Big\}\\
&=A_1+M\|x\|_\beta^\infty\Big\{1+\frac{L}{\beta}\Big(1-\frac{1}{e_\beta(b,a)}\Big)+\frac{N}{\beta}\Big(e_\beta(b,a)-1\Big)\Big\}\\
&<A_1+M\|x\|_\beta^\infty\Big\{1+\frac{1}{\beta}(L+Ne_\beta(b,a))\Big\}\\
&=A_1+M\|x\|_\beta^\infty\Big\{1+\frac{1}{\gamma}\Big\}<\infty.
\end{align*}
This proves that the operator $F$ maps $C_\beta(I_\mathbb{T};\mathbb{X})$ into itself. Now let $u,v\in C_\beta(I_\mathbb{T};\mathbb{X})$. From the hypotheses, we have
\begin{align*}
\|Fu-Fv\|_{\beta}^\infty&=\sup_{t\in I_\mathbb{T}}\frac{\|(Fu)(t)-(Fv)(t)\|}{e_\beta(t,a)}\\
&=\sup_{t\in I_\mathbb{T}}\frac{1}{e_\beta(t,a)}\Big\|f(t,u(t),\int_a^th(t,s,u(s))\Delta s,\int_a^bg(t,s,u(s))\Delta s\\
&\qq \q -f(t,v(t),\int_a^th(t,s,v(s))\Delta s,\int_a^bg(t,s,v(s))\Delta s\Big\|\\
&\leq \sup_{t\in I_\mathbb{T}}\frac{1}{e_\beta(t,a)}M\Big\{\|u(t)-v(t)\|+L\int_a^t\|u(s)-v(s)\|\Delta s+N\int_a^b\|u(s)-v(s)\|\Delta s\Big\}\\
&=M\Big\{\sup_{t\in I_\mathbb{T}}\frac{\|u(t)-v(t)\|}{e_\beta(t,a)}+\sup_{t\in I_\mathbb{T}}\frac{1}{e_\beta(t,a)}L\int_a^te_\beta(s,a)\frac{\|u(s)-v(s)\|}{e_\beta(s,a)} \Delta s  \\
&\qq \q+\sup_{t\in I_\mathbb{T}}\frac{1}{e_\beta(t,a)}N\int_a^be_\beta(s,a)\frac{\|u(s)-v(s)\|}{e_\beta(s,a)} \Delta s \Big\}\\
&\leq M\Big\{\|u-v\|_{\beta}^\infty+L\|u-v\|_{\beta}^\infty\sup_{t\in I_\mathbb{T}}\frac{1}{e_\beta(t,a)}\int_a^te_\beta(s,a)\Delta s\\
&\qq \q +N\|u-v\|_{\beta}^\infty\sup_{t\in I_\mathbb{T}}\frac{1}{e_\beta(t,a)}\int_a^be_\beta(s,a)\Delta s\Big\}\\
&=M\|u-v\|_{\beta}^\infty\Big\{1+L\sup_{t\in I_\mathbb{T}}\frac{1}{e_\beta(t,a)}\Big( \frac{e_\beta(t,a)-1}{\beta}\Big)+N\sup_{t\in I_\mathbb{T}}\frac{1}{e_\beta(t,a)}\Big(\frac{e_\beta(b,a)-1}{\beta} \Big)\Big\}\\
&\leq M\|u-v\|_{\beta}^\infty\Big\{1+\frac{L}{\beta}\Big(1-\frac{1}{e_\beta(b,a)}\Big)+\frac{N}{\beta}\Big(e_\beta(b,a)-1\Big)\Big\}\\
&\leq M\|u-v\|_{\beta}^\infty\Big\{1+\frac{1}{\beta}(L+Ne_\beta(b,a))\Big\}\qq \\
&=M\Big\{1+\frac{1}{\gamma}\Big\}\|u-v\|_{\beta}^\infty.
\end{align*}
Since $M(1+\frac{1}{\gamma})<1$, it follows from the Banach's fixed point theorem that $F$ has a unique fixed point in $C_\beta(I_\mathbb{T};\mathbb{X})$. The fixed point of $F$ is the unique solution $x$ of equation \eqref{e1}.
\end{proof}



\begin{thm}
In addition to the assumptions \eqref{e2}-\eqref{e} from Theorem \ref{th1} assume that
\begin{equation*}\label{e5}
 A_2:=\sup_{t\in I_\mathbb{T}}\frac{1}{e_\beta(t,a)}\Big\|\int_a^tf\Big(\tau,0,\int_a^\tau h(\tau,s,0)\Delta s,\int_a^bg(\tau,s,0)\Delta s\Big)\Delta\tau\Big\|<\infty.
\end{equation*}
Then the integro-dynamic equation \eqref{e.2} has a unique solution  $x\in C_\beta(I_\mathbb{T};\mathbb{X})$.
\end{thm}
\begin{proof}
The corresponding integral equation to \eqref{e.2} is
$$
x(t)=x(a)+\int_a^tf\Big(\tau,x(\tau),\int_a^\tau h(\tau,s,x(s))\Delta s,\int_a^bg(\tau,s,x(s))\Delta s\Big)\Delta \tau.
$$
Define the operator $F:C_\beta(I;\mathbb{X})\rightarrow C_\beta(I;\mathbb{X})$ by
$$
[Fx](t):= x(a)+\int_a^tf\Big(\tau,x(\tau),\int_a^\tau h(\tau,s,x(s))\Delta s,\int_a^bg(\tau,s,x(s))\Delta s\Big)\Delta \tau,
$$


for $t\in I_\mathbb{T}$. By following the same argument as in proof of Theorem \ref{th1}, we can similarly prove the existence and uniqueness of solutions of equation \eqref{e.2} on $I_\mathbb{T}$.
\end{proof}


\section{Estimate on solutions}

The following two theorems provide a certain estimate on the solutions of equations \eqref{e1} and \eqref{e.2}, respectively.
\begin{thm}
Consider the nonlinear dynamic integral equation \eqref{e1}. Suppose that there exist positive constants $\alpha, B, C,$ and $0\leq N<1$ such that
\begin{equation}\label{e12}
    \|f(t,u,v,w)-f(t,\bar{u},\bar{v},\bar{w})\|\leq N\{\|u-\bar{u}\|+\|v-\bar{v}\|+\|w-\bar{w}\|\},
\end{equation}
\begin{equation}\label{e14}
  \|h(t,s,v)-h(t,s,\bar{v})\|\leq Be_\alpha(s,t)\|v-\bar{v}\|,
\end{equation}
\begin{equation}\label{e13}
  \|g(t,s,u)-g(t,s,\bar{u})\|\leq Ce_\alpha(s,t)\|u-\bar{u}\|,
\end{equation}
\begin{equation}\label{condition}
  q:=\int_a^bCN^*e_{B{N}^*}(s,a)\Delta s<1 ,\q N^*=\frac{N}{1-N}\cdot
\end{equation}
hold. Moreover, assume that there exists a nonnegative constant $d$ such that
\begin{equation}\label{e15}
  \Big\|f(t,0,\int_a^th(t,s,0)\Delta s,\int_a^bg(t,s,0)\Delta s)\Big\|\leq \frac{d}{e_\alpha(t,a)}, \q t \in I_\mathbb{T}.
\end{equation}
If $x(t)$ is any solution of equation \eqref{e1}, then
\begin{equation}\label{16}
  \|x(t)\|\leq \frac{d}{(1-N)(1-q)}e_{BN^*\ominus\alpha}(t,a), \q t\in I_\mathbb{T}.
\end{equation}
\end{thm}
\begin{proof}
For any solution $x(t)$ of \eqref{e1} on $I_\mathbb{T}$, we have
\begin{align*}
\|x(t)\|&\leq \Big\|f\Big(t,0,\int_a^t h(t,s,0)\Delta s,\int_a^b g(t,s,0)\Delta s\Big)\Big\|\\
&\qq +\Big\|f\Big(t,x(t),\int_a^t h(t,s,x(s))\Delta s,\int_a^b g(t,s,x(s))\Delta s\Big)\\
&\qq \qq -f\Big(t,0,\int_a^t h(t,s,0)\Delta s,\int_a^b g(t,s,0)\Delta s\Big)\Big\|.
\end{align*}
From the assumptions we see that
$$
\|x(t)\|\leq \frac{d}{e_\alpha(t,a)}+N\Big[\|x(t)\|+\int_a^tBe_\alpha(s,t)\|x(s)\|\Delta s + \int_a^bCe_\alpha(s,t)\|x(s)\|\Delta s\Big], \q t\in I_\mathbb{T}.
$$
Multiplying both sides of the above inequality by $e_{\alpha}(t,a)$ and rearranging the terms we observe that
$$
e_\alpha(t,a)\|x(t)\|\leq \frac{d}{1-N}+\int_a^tBN^*e_\alpha(s,a)\|x(s)\|\Delta s+\int_a^bCN^*e_\alpha(s,a)\|x(s)\|\Delta s,
$$
for $t\in I_\mathbb{T}$, where we have used the semigroup property
$$
e_\alpha(s,a)=e_\alpha(s,t)e_\alpha(t,a).
$$
Now, in view of condition \eqref{condition} and applying Theorem \ref{th3.1} we obtain
$$
e_\alpha(t,a)\|x(t)\|\leq \frac{d}{(1-N)(1-q)}e_{BN^*}(t,a), \q t\in I_\mathbb{T}.
$$
By using the identity \cite[Theorem 2.36]{Boh1}, we deduce that
\begin{equation*}
\|x(t)\|\leq \frac{d}{(1-N)(1-q)}e_{BN^*\ominus\alpha}(t,a), \q t\in I_\mathbb{T}.
\end{equation*}
\end{proof}
\begin{thm}\label{integro}
Consider the integro-dynamic equation \eqref{e.2}. Assume that there are functions $k_1, k_2, k_3 \in C_{rd}(I_\mathbb{T};\mathbb{R}_+)$ such that
\begin{equation}\label{e12.0}
  \|f(t,u,v,w)-f(t,\bar{u},\bar{v},\bar{w})\|\leq k_1(t)\{\|u-\bar{u}\|+\|v-\bar{v}\|+\|w-\bar{w}\|\},
\end{equation}
\begin{equation}\label{e14.0}
  \|h(t,s,v)-h(t,s,\bar{v})\|\leq k_2(t)\|v-\bar{v}\|,
\end{equation}
\begin{equation}\label{e13.0}
  \|g(t,s,u)-g(t,s,\bar{u})\|\leq k_3(t)\|u-\bar{u}\|,
\end{equation}
\begin{equation}\label{e14.00}
 r:=\int_a^bk_3(\tau)e_{k_1+k_2}(\tau,a)\Delta\tau <1,
\end{equation}
hold.
If $x(t)$ is any solution of \eqref{e.2} on $I_\mathbb{T}$, then
\begin{equation}\label{e16.0}
  \|x(t)\|\leq \frac{m}{1-r}e_{k_1+k_2}(t,a), \q t\in I_\mathbb{T}.
\end{equation}
where
\begin{equation}\label{e15.0}
  m:=\sup_{t\in I_\mathbb{T}}\Big\|x(a)+\int_a^t f\Big(\tau,0,\int_a^\tau h(\tau,s,0)\Delta s,\int_a^bg(\tau,s,0)\Delta s\Big)\Delta\tau\Big\|<\infty.
\end{equation}
\end{thm}
\begin{proof}
For any solution $x(t)$ of \eqref{e.2} on $I_\mathbb{T}$, we have
\begin{align*}
\|x(t)\|&\leq \Big\|x(a)+\int_a^t f\Big(\tau,0,\int_a^\tau h(\tau,s,0)\Delta s,\int_a^bg(\tau,s,0)\Delta s\Big)\Delta\tau\Big\|\\
&\qq +\Big\|\int_a^tf\Big(\tau,x(\tau),\int_a^\tau h(\tau,s,x(s))\Delta s,\int_a^b g(\tau,s,x(s))\Delta s\Big)\\
&\qq \qq-f\Big(\tau,0,\int_a^\tau h(\tau,s,0)\Delta s,\int_a^b g(\tau,s,0)\Delta s\Big)\Delta\tau\Big\|, \q t\in I_\mathbb{T}.
\end{align*}
From the assumptions we get
$$
\|x(t)\|\leq m +\int_a^tk_1(\tau)\Big[\|x(\tau)\|+\int_a^\tau k_2(\tau)\|x(s)\|\Delta s + \int_a^bk_3(\tau)\|x(s)\|\Delta s\Big]\Delta\tau, \q t\in I_\mathbb{T}.
$$
Now applying Theorem \ref{th3.2} to the above inequality, we get the desired estimate.
\end{proof}


\section{Continuous Dependence of Solutions}
In this section, we are interested in estimating the change in the solution for equations \eqref{e1} and \eqref{e.2} respectively, when the function $f$, $h$, $g$ are allowed to change. Besides equation \eqref{e1} we consider the perturbed equation
\begin{equation}\label{19}
    y(t)=\bar{f}\Big(t,y(t),\int_a^t \bar{h}(t,s,y(s))\Delta s,\int_a^b \bar{g}(t,s,y(s))\Delta s\Big),
\end{equation}
for $t\in I_\mathbb{T}:=[a,\infty)_\mathbb{T}$, where $\bar{f}:I_\mathbb{T}\times \mathbb{X}\times \mathbb{X}\times \mathbb{X}\rightarrow \mathbb{X}$ is rd-continuous in its first variable, while the functions $\bar{h}:I_\mathbb{T}^2\times \mathbb{X}\rightarrow \mathbb{X}$ and $\bar{g}:I_\mathbb{T}^2\times \mathbb{X}\rightarrow \mathbb{X}$, are rd-continuous in its second variable. Let $x(t)$ and $y(t)$ be the solutions of equations \eqref{e1} and \eqref{19} respectively. We answer the following question:

Under what conditions does the solution $x(t)$ of equation \eqref{e1} depends continuously on the functions involved $f, h$, and $g$ ?
\begin{thm}\label{th6.1}
Consider the integral equation \eqref{e1}. Suppose that there are two functions $n,m:I_\mathbb{T}^2\rightarrow \mathbb{R}$ which are rd-continuous in its second variable and a constant $0\leq N < 1$ such that
\begin{equation}\label{e20}
    \|f(t,u,v,w)-f(t,\bar{u},\bar{v},\bar{w})\|\leq N\{\|u-\bar{u}\|+\|v-\bar{v}\|+\|w-\bar{w}\|\},
\end{equation}
\begin{equation}\label{e22}
  \|h(t,s,v)-h(t,s,\bar{v})\|\leq m(t,s)\|v-\bar{v}\|,
\end{equation}
\begin{equation}\label{e21}
  \|g(t,s,u)-g(t,s,\bar{u})\|\leq n(t,s)\|u-\bar{u}\|,
\end{equation}
\begin{equation}\label{con}
  p(t):=N^*\int_a^bn(t,s)e_{m(s,.)}(s,a)\Delta s<1 , \q N^*=\frac{N}{1-N}\CommaBin
\end{equation}
hold. In addition, assume that $\varepsilon$ is arbitrary positive number, and $y(t)$ is a solution of \eqref{19} such that
\begin{align*}
 R(t)&=\Big\|f\Big(t,y(t),\int_a^t h(t,s,y(s))\Delta s,\int_a^b g(t,s,y(s))\Delta s\Big)\\
 &\qq -\bar{f}\Big(t,y(t),\int_a^t \bar{h}(t,s,y(s))\Delta s,\int_a^b \bar{g}(t,s,y(s))\Delta s\Big)\Big\|<\varepsilon, \q t\in I_\mathbb{T},
\end{align*}
where $f,$ $h,$ $g$ and $\bar{f},$ $\bar{h},$ $\bar{g}$ are the functions involved in \eqref{e1} and \eqref{19} respectively. Then every solution $x(t)$ of \eqref{e1} satisfies
\begin{equation}\label{dependence1}
  \|x(t)-y(t)\|\leq \frac{\varepsilon}{(1-N)(1-p(t))}e_{N^*m(t,.)}(t,a),\q  t\in I_\mathbb{T}.
\end{equation}
Furthermore if $I_\mathbb{T}=[a,b]_\mathbb{T}$, then every solution of equation \eqref{e1} depends continuously on the functions $f,g,h$ involved.
\end{thm}
\begin{proof}
Let $x(t)$ be a solution of \eqref{e1}, and $y(t)$ be a solution of the perturbed equation \eqref{19}. Put $u(t):=\|x(t)-y(t)\|$, $t\in I_\mathbb{T}$. We obtain
\begin{align*}
u(t)&\leq \Big\|f\Big(t,x(t),\int_a^t h(t,s,x(s))\Delta s,\int_a^b g(t,s,x(s))\Delta s\Big)\\
&\qq - f\Big(t,y(t),\int_a^t h(t,s,y(s))\Delta s,\int_a^b g(t,s,y(s))\Delta s\Big)\Big\|\\
 &\qq + \Big\|f\Big(t,y(t),\int_a^t h(t,s,y(s))\Delta s,\int_a^b g(t,s,y(s))\Delta s\Big)\\
 &\qq -\bar{f}\Big(t,y(t),\int_a^t \bar{h}(t,s,y(s))\Delta s,\int_a^b \bar{g}(t,s,y(s))\Delta s\Big)\Big\|\\
 & \leq \varepsilon+N\Big[u(t)+ \int_a^t m(t,s)u(s)\Delta s + \int_a^b n(t,s)u(s)\Delta s\Big].
\end{align*}
Using the assumption $0\leq N< 1$, we get
$$
u(t)\leq \frac{\varepsilon}{1-N}+\int_a^t N^*m(t,s)u(s)\Delta s +\int_a^b N^*n(t,s)u(s)\Delta s.
$$
Now in view of Theorem \ref{th3.1}, the previous inequality yields
$$
u(t):=\|x(t)-y(t)\|\leq \frac{\varepsilon}{(1-N)(1-p(t))}e_{N^*m(t,.)}(t,a), \q t\in I_\mathbb{T}.
$$
The function $e_{N^*m(t,.)}(t,a)$ is rd-continuous function on the compact interval $I_\mathbb{T}=[a,b]_\mathbb{T}$, so it is a bounded function. Therefore, the solution $x(t)$ of equation \eqref{e1} depends continuously on the functions involved $f$, $g$, $h$. 
\end{proof}\s

Next, we introduce a result on continuous dependence to solutions of equation \eqref{e.2} on the functions involved $f$, $g$, $h$. Consider the perturbed equation
\begin{equation}\label{pe}
  y(t)=y(a)+\int_a^t\tilde{f}\Big(\tau,y(\tau),\int_a^\tau \tilde{h}(\tau,s,y(s))\Delta s,\int_a^b\tilde{g}(\tau,s,y(s))\Delta s\Big)\Delta \tau,
\end{equation}
for $t\in I_\mathbb{T}:=[a,\infty)_\mathbb{T}$ where $\tilde{f}:I_\mathbb{T}\times \mathbb{X}\times \mathbb{X}\times \mathbb{X}\rightarrow \mathbb{X}$ is rd-continuous in its first variable, while the functions $\tilde{h}:I_\mathbb{T}^2\times \mathbb{X}\rightarrow \mathbb{X}$ and $\tilde{g}:I_\mathbb{T}^2\times \mathbb{X}\rightarrow \mathbb{X}$, are rd-continuous in its second variable.

\begin{thm}
Assume that there are $k_1,k_2,k_3\in C_{rd}(I_\mathbb{T};\mathbb{R}_+)$ such that conditions \eqref{e12.0}-\eqref{e14.00} hold. If $x$ and $y$ are solutions of equations \eqref{e.2} and \eqref{pe} respectively that satisfy
\begin{align*}
 P(t):&=\|x(a)-y(a)\|+\int_a^t\Big\|f\Big(\tau,x(\tau),\int_a^\tau h(\tau,s,x(s))\Delta s,\int_a^bg(\tau,s,x(s))\Delta s\Big)\\
 & \qq -\tilde{f}\Big(\tau,x(\tau),\int_a^\tau \tilde{h}(\tau,s,x(s))\Delta s,\int_a^b\tilde{g}(\tau,s,x(s))\Delta s\Big)\Big\|\Delta \tau< \varepsilon,
 \end{align*}
then
$$
\|x(t)-y(t)\|\leq\frac{\varepsilon}{1-r}e_{k_1+k_2}(t,a), \q t\in I_\mathbb{T}.
$$
\end{thm}
\begin{proof}
Let $x(t)$ be a solution of \eqref{e.2}, and $y(t)$ be a solution of the perturbed equation \eqref{pe}. Put $u(t):=\|x(t)-y(t)\|$, $t\in I_\mathbb{T}$. We have
\begin{align*}
u(t)&\leq \|x(a)-y(a)\|+\int_a^t\Big\|f\Big(\tau,x(\tau),\int_a^\tau h(\tau,s,x(s))\Delta s,\int_a^bg(\tau,s,x(s))\Delta s\Big)\\
  &-f\Big(\tau,y(\tau),\int_a^\tau h(\tau,s,y(s))\Delta s,\int_a^bg(\tau,s,y(s))\Delta s\Big)\Big\|\Delta\tau \\
  &+\int_a^t \Big\|f\Big(\tau,y(\tau),\int_a^\tau h(\tau,s,y(s))\Delta s,\int_a^b g(\tau,s,y(s))\Delta s\Big)\\
  &-\tilde{f}\Big(t,y(t),\int_a^\tau \tilde{h}(\tau,s,y(s))\Delta s,\int_a^b \tilde{g}(\tau,s,y(s))\Delta s\Big)\Big\|\Delta\tau\\
  & \leq \varepsilon+\int_a^tk_1(\tau)\Big[u(\tau)+\int_a^\tau k_2(\tau)u(s)\Delta s + \int_a^bk_3(\tau)u(s)\Delta s\Big]\Delta\tau.
\end{align*}
Applying Theorem \ref{th3.2}, we obtain
$$
u(t)\leq\frac{\varepsilon}{1-r}e_{k_1+k_2}(t,a), \q t\in I_\mathbb{T}.
$$
\end{proof}
\begin{cor}
Assume that there are $k_1,k_2,k_3\in C_{rd}(I_\mathbb{T};\mathbb{R}_+)$ such that conditions \eqref{e12.0}-\eqref{e14.00} hold. If $I_\mathbb{T}=[a,b]_\mathbb{T}$ is a compact interval, then every solution of equation \eqref{e.2} depends continuously on the functions $f,g,h$ involved.
\end{cor}
\begin{proof}
Since $e_{k_1+k_2}(t,a)$ is rd-continuous function on the compact interval $I_\mathbb{T}=[a,b]_\mathbb{T}$, so it is bounded function. Therefore, every solution $x(t)$ of equation \eqref{e.2} depends continuously on the functions involved $f$, $g$, $h$. 
\end{proof}

\bibliographystyle{srtnumbered}

\end{document}